\documentclass[11pt,a4paper]{article}
\usepackage{amsmath,amsthm,amssymb,mathrsfs}

\makeatletter
\oddsidemargin\paperwidth
\advance\oddsidemargin by-\textwidth
\advance\oddsidemargin by-1in
\evensidemargin\oddsidemargin
\headsep\footskip
\advance\headsep by-\topskip
\advance\textheight by\topskip
\topmargin\paperheight
\advance\topmargin by-\textheight
\advance\topmargin by-1in
\advance\topmargin by-\headheight
\advance\topmargin by-\headsep

\renewcommand{\section}{%
	\@startsection{section}{1}{\z@}%
	{3.5ex \@plus 1ex \@minus .2ex}%
	{2.3ex \@plus.2ex}%
	{\normalfont\large\bfseries\centering}}
\renewcommand{\subsection}{%
	\@startsection{subsection}{2}{\z@}%
	{3.25ex\@plus 1ex \@minus .2ex}%
	{.8ex \@plus .2ex}%
	{\normalfont\bfseries}}

\renewcommand{\@seccntformat}[1]{\S\csname the#1\endcsname.\ \,}

\renewenvironment{proof}[1][\proofname]{\par
	\pushQED{\qed}%
	\normalfont \topsep6\p@\@plus6\p@\relax
	\trivlist
	\item[%
	\scshape
	\hspace*{\parindent}
	#1\@addpunct{.}]\ignorespaces
}{%
	\popQED\endtrivlist\@endpefalse
}

\def\@endtheorem{\endtrivlist}
\newtheoremstyle{mythm}{}{}{\itshape}{\parindent}{\bfseries}{.}{5.555pt plus 1.666pt minus 1.666pt}{}
\newtheoremstyle{myrem}{}{}{\normalfont}{\parindent}{\scshape}{.}{5.555pt plus 1.666pt minus 1.666pt}{}

\renewenvironment{abstract}%
{%
	\small
	\list{\scshape\abstractname.}%
	{%
		\labelsep 7pt
		\itemindent \labelwidth
		\advance \itemindent by \labelsep
		\listparindent 0pt
		\leftmargin 0.1\paperwidth
		\rightmargin\leftmargin
	\parsep 0pt}%
\item\relax}%
{\endlist}
\renewcommand{\abstractname}{Abstract}

\let\@oldenumerate=\enumerate
\def\enumerate{\@oldenumerate\def\makelabel##1{\hss ##1}}
\let\@olditemize=\itemize
\def\itemize{\@olditemize\def\makelabel##1{\hss ##1}}

\@addtoreset{equation}{section}

\def\@address{}
\def\@email{}
\def\@msc{}
\def\@mscyear{}
\def\@headtitle{}
\def\@headauthor{}
\renewcommand{\title}[2][]{%
	\gdef\@title{#2}%
	\gdef\@headtitle{#1}}
\renewcommand{\author}[2][]{%
	\gdef\@author{#2}%
	\gdef\@headauthor{#1}}
\newcommand*{\address}[1]{\gdef\@address{#1}}
\newcommand*{\email}[1]{\gdef\@email{#1}}
\newcommand*{\subjclass}[2][2000]{\gdef\@mscyear{#1}\gdef\@msc{#2}}

\def\@org{%
	\footnotetext{%
		\hspace*{-1em}\hspace*{-\footnotesep}%
		\ifx\@address\@empty\else\hspace*{\footnotesep}\textsc{\@address}\\\fi
		\ifx\@email\@empty\else\hspace*{\footnotesep}\textit{E-mail address}: \texttt{\@email}\\\fi
\ifx\@msc\@empty\else\hspace*{\footnotesep}{\@mscyear} \textit{Mathematics Subject Classification}. \@msc.\fi}}

\let\@old@@maketitle=\@maketitle
\def\@maketitle{%
	\@old@@maketitle
	\ifx\@address\@empty
	\ifx\@email\@empty
	\ifx\@msc\@empty
	\else\@org\fi
	\else\@org\fi
	\else\@org\fi
}
\let\@old@maketitle=\maketitle
\def\maketitle{%
	\ifx\@headtitle\@empty\let\@headtitle=\@title\fi
	\ifx\@headauthor\@empty\let\@headauthor=\@author\fi
	\@old@maketitle
\gdef\@address{}\gdef\@email{}\gdef\@mscyear{}\gdef\@msc{}}

\def\ps@mystyle{%
	\def\@oddhead{\def\\{\relax}\reset@font\small\hfil\ifodd\value{page}\@headtitle\else\@headauthor\fi\hfil}%
\def\@oddfoot{\reset@font\hfil\thepage\hfil}}
\makeatother

\renewcommand{\labelenumi}{(\theenumi)}
\theoremstyle{mythm}
\newtheorem{thm}{Theorem}[section]

\newtheorem{lem}[thm]{Lemma}

\theoremstyle{myrem}

\DeclareMathOperator{\Hom}{Hom}
\DeclareMathOperator{\Ext}{Ext}
\DeclareMathOperator{\Ker}{Ker}

\DeclareMathOperator{\id}{id}

\newcommand*{\Z}{\mathbb{Z}}
\newcommand*{\C}{\mathbb{C}}
\newcommand*{\R}{\mathbb{R}}

\title[Homomorphisms between principal series]{On the existence of homomorphisms\\ between principal series of complex semisimple Lie groups}
\author{Noriyuki Abe}
\address{Graduate School of Mathematical Sciences, the University of Tokyo, 3--8--1 Komaba, Meguro-ku, Tokyo 153--8914, Japan.}
\email{abenori@ms.u-tokyo.ac.jp}
\subjclass{Primary 22E47, Secondary 17B10}
\date{}
\begin{document}
\maketitle

\begin{abstract}
We determine when there exists a nonzero homomorphism between principal series representations of a complex semisimple Lie group.
We also determine the condition for the existence of nonzero homomorphisms between twisted Verma modules.
\end{abstract}

\section{Introduction}\label{sec:Introduction}
Let $G$ be a complex semisimple Lie group.
Then the principal series representations of $G$ are defined and play an important role in the representation theory of $G$.
One of a fundamental problem about principal series is a description of the space of homomorphisms between such representations (cf.~\cite[p.~720, II]{MR0435300}).
In this paper, we determine when there exists a nonzero homomorphism between principal series representations of a complex semisimple Lie group.
We also determines the existence of homomorphisms between twisted Verma modules.
This gives a generalization of results of Verma~\cite{MR0218417} and Bernstein-Gelfand-Gelfand~\cite{MR0291204}.

We state our main results.
Let $\mathfrak{g}$ be the Lie algebra of $G$, $\mathfrak{h}$ its Cartan subalgebra, $\Delta$ the root system for $(\mathfrak{g},\mathfrak{h})$ and $W$ the Weyl group of $\Delta$.
By the Killing form we identify $\mathfrak{g}$ with $\mathfrak{g}^* = \Hom_\C(\mathfrak{g},\C)$.
Then the Killing form also defines a non-degenerate bilinear form on $\mathfrak{g}^*$.
We denote this form by $\langle\cdot,\cdot\rangle$.
For $\alpha\in \mathfrak{h}^*$, put $\check{\alpha} = 2\alpha/\langle\alpha,\alpha\rangle$ and $s_\alpha(\lambda) = \lambda - \langle\check{\alpha},\lambda\rangle\alpha$.
Take a positive system $\Delta^+\subset \Delta$.
Then $\Delta^+$ determines a Borel subalgebra $\mathfrak{b}$.
Put $\mathfrak{n} = [\mathfrak{b},\mathfrak{b}]$.
Let $\mathcal{O}$ be the Bernstein-Gelfand-Gelfand category~\cite[Definition~1]{MR0407097} for $(\mathfrak{g},\mathfrak{b})$ and $M(\lambda)$ the Verma module with highest weight $\lambda - \rho$ for $\lambda\in\mathfrak{h}^*$ where $\rho$ is the half sum of positive roots.
Fix an involution $\sigma$ of $\mathfrak{g}$ such that $\sigma|_\mathfrak{h} = -\id_\mathfrak{h}$.
The category $\mathcal{O}$ has a dualizing functor $\delta$ defined by $\delta M = \Hom_\C(M,\C)_{\text{$\mathfrak{h}$-finite}}$ where the action is given by $(Xf)(m) = f(-\sigma(X)m)$.
Put $\mathfrak{k} = \{(X,\sigma(X))\mid X\in \mathfrak{g}\}\subset \mathfrak{g}\oplus\mathfrak{g}$.
For $M,N\in\mathcal{O}$, we define the $\mathfrak{g}\oplus\mathfrak{g}$-module $L(M,N) = \Hom_\C(M,N)_{\text{$\mathfrak{k}$-finite}}$ where the action is given by $((X,Y)f)(m) = \sigma(X)f(-Ym)$.
Then under some identification $\mathfrak{g}\otimes_\R\C\simeq \mathfrak{g}\oplus\mathfrak{g}$, the principal representations of $G$ are $L(\lambda,\mu) = L(M(-\mu),\delta M(-\lambda))$.
This is an object of $\mathcal{H}$ where $\mathcal{H}$ is a category of Harish-Chandra modules.

For $\lambda\in\mathfrak{h}^*$, let $\Delta_\lambda$ be the integral root system of $\lambda$, $W_\lambda$ the Weyl group of $\Delta_\lambda$.
Let $\mathcal{P}$ be the integral weight lattice of $\Delta$.
Then it is well-known that $W_\lambda = \{w\in W\mid w\lambda - \lambda\in\mathcal{P}\}$.
Let $w_\lambda$ be the longest element of $W_\lambda$.
Put $\Delta_\lambda^+ = \Delta^+\cap \Delta_\lambda$.
Then $\Delta_\lambda^+$ determines the set of simple roots $\Pi_\lambda$.
Put $S_\lambda = \{s_\alpha\mid\alpha\in\Pi_\lambda\}$ and $W_\lambda^0 = \{w\in W_\lambda\mid w\lambda = \lambda\}$.
For $w\in W_\lambda$, let $\ell_\lambda(w)$ be a length of $w$ as an element of $W_\lambda$.

For a sequence of simple roots $\alpha_1,\dots,\alpha_l\in \Pi_\lambda$ and $\mu\in\mathfrak{h}^*$, we define a subset $A_{(s_{\alpha_1},\dots,s_{\alpha_l})}(\mu)$ of $\mathfrak{h}^*$ as follows.
Put $\beta_i = s_{\alpha_1}\dotsm s_{\alpha_{i - 1}}(\alpha_i)$ for $i = 1,\dots,l$.
For $\mu\in \mathfrak{h}^*$, put
\[
	A_{(s_{\alpha_1},\dots,s_{\alpha_l})}(\mu) =
	\left\{
	\mu'\in\mathfrak{h}^*\Biggm|
	\begin{array}{lll}
	\text{for some $1\le i_1 < \dots < i_r\le l$, $\mu' = s_{\beta_{i_r}}\dotsm s_{\beta_{i_1}}\mu$ and}\\
	\text{$\langle \check{\beta_{i_k}},s_{\beta_{i_{k - 1}}}\dotsm s_{\beta_{i_1}}\mu\rangle\in\Z_{<0}$ for all $k = 1,\dots,r$}
	\end{array}
	\right\}
\]
For a reduced expression $w = s_1\dotsm s_l\in W$, it will be proved that the set $A_{(s_1,\dots,s_l)}(\mu)$ is independent of the choice of a reduced expression (Lemma~\ref{lem:from verma mdoule}).
We write $A_w(\mu)$ instead of $A_{(s_1,\dots,s_l)}(\mu)$.

Now we state the main theorems of this paper.

\begin{thm}\label{thm:main theorem, principal series}
Let $\lambda\in\mathfrak{h}^*$, $\mu_1,\mu_2\in\lambda + \mathcal{P}$ and $w,w'\in W_\lambda$.
Assume that $\lambda$ is dominant, i.e., $\langle\check{\alpha},\lambda\rangle\not\in\Z_{<0}$ for all $\alpha\in\Delta^+$.
Then $\Hom_\mathcal{H}(L(M(w_1\lambda),\delta M(\mu_1)),L(M(w_2\lambda),\delta M(\mu_2))) \ne 0$ if and only if $w_1^{-1}w_\lambda A_{w_\lambda w_1}(w_\lambda\mu_1)\cap W_\lambda^0w_2^{-1}A_{w_2}(\mu_2)\ne \emptyset$.

Moreover, if $\Hom_\mathcal{H}(L(M(w_1\lambda),\delta M(\mu_1)),L(M(w_2\lambda),\delta M(\mu_2))) = 0$, then for all $k\in\Z_{\ge 0}$, we have $\Ext^k_\mathcal{H}(L(M(w_1\lambda),\delta M(\mu_1)),L(M(w_2\lambda),\delta M(\mu_2))) = 0$.
\end{thm}

We can determine when there exists a nonzero homomorphisms between principal series representations of $G$ from Theorem~\ref{thm:main theorem, principal series} (see Lemma~\ref{lem:isom between principal series}).

Let $T_w$ be the twisting functor for $w\in W$~\cite[6.2]{MR1985191} and $w_0$ the longest element of $W$ (see also Arkhipov~\cite[Definition~2.3.4]{MR2074588}).
\begin{thm}\label{thm:main theorem, twisted Verma modules}
We have $\Hom_\mathcal{O}(T_{w_1}M(\mu_1),T_{w_2}M(\mu_2))\ne 0$ if and only if $w_1 A_{w_1^{-1}}(\mu_1)\cap w_2w_0 A_{w_0w_2^{-1}}(w_0\mu_2)\ne \emptyset$.

Moreover, if $\Hom_\mathcal{O}(T_{w_1}M(\mu_1),T_{w_2}M(\mu_2)) = 0$, then $\Ext^k_\mathcal{O}(T_{w_1}M(\mu_1),T_{w_2}M(\mu_2)) = 0$ for all $k\in\Z_{\ge 0}$.
\end{thm}
The proof of this theorem gives a new proof of the famous result of Verma~\cite{MR0218417} and Bernstein-Gelfand-Gelfand~\cite{MR0291204} about homomorphisms between Verma modules.

\subsection*{Acknowledgments}
The author would like to thank his advisor Hisayosi Matumoto for reading a manuscript and giving comments.
He is supported by the Japan Society for the Promotion of Science Research Fellowships for Young Scientists.

\section{General theory}\label{sec:General theory}
We use the notation in Section~\ref{sec:Introduction}.
It is easy to prove the following lemma.
We omit the proof.
\begin{lem}\label{lem:sum of the set A}
Let $s_1,\dots,s_l,s_1',\dots,s_{l'}'\in S_\lambda$ be simple reflections.
Put $w = s_1\dotsm s_l$.
Then we have $A_{(s_1,\dots,s_l,s_1',\dots,s_{l'}')}(\mu) = \bigcup_{\mu'\in A_{(s_1,\dots,s_l)}(\mu)} wA_{(s_1',\dots,s_{l'}')}(w^{-1}\mu')$.
\end{lem}

Fix a dominant $\lambda\in\mathfrak{h}^*$.
Let $\mathcal{C}$ be an abelian category with enough injective objects, $\mathcal{D}\subset \mathfrak{h}^*$ a $W_\lambda$-stable subset.
Let $\{M_\lambda(w,\mu)\mid w\in W,\mu\in\mathcal{D}\}$ be objects of $\mathcal{C}$ such that the following conditions are satisfied:
\begin{enumerate}
\renewcommand*{\labelenumi}{(A\theenumi)}
\item For $w\in W_\lambda$ and $w'\in W_{\lambda}^0$, $M_\lambda(ww',\mu) \simeq M_\lambda(w,\mu)$.\label{enum:w <-> w lambda condition}
\item For $\alpha\in\Pi_\lambda$ such that $s_\alpha w > w$, if $\langle\check{\alpha},\mu\rangle\not\in\Z_{<0}$ then we have $M_\lambda(s_\alpha w,\mu)\simeq M_\lambda(w,s_\alpha\mu)$.\label{enum:isom condition}
\item For $\alpha\in\Pi_\lambda$ such that $s_\alpha w > w$, $\langle\alpha,w\lambda\rangle \ne 0$ and $\langle\check{\alpha},\mu\rangle\in\Z_{<0}$ there exists an exact sequence $0\to M_\lambda(w,\mu)\to M_\lambda(s_\alpha w,\mu)\to M_\lambda(w,s_\alpha\mu)\to M_\lambda(w,\mu)\to 0$.\label{enum:exact sequence condition}
\item We have $\Hom_\mathcal{C}(M_\lambda(w_\lambda,\mu'),M_\lambda(e,\mu))\ne 0$ if and only if $\mu \in W_\lambda^0w_\lambda \mu'$.\label{enum:basis step conditoin - Hom}
\item We have $\Ext^k_\mathcal{C}(M_\lambda(w_\lambda,\mu'),M_\lambda(e,\mu)) = 0$ for $k > 0$.\label{enum:basis step conditoin - Ext}
\end{enumerate}

\begin{lem}\label{lem:<alpha wlambda> = 0}
Let $\alpha\in\Pi_\lambda$, $w\in W_\lambda$, $\mu\in\mathcal{D}$.
Assume that $\langle\alpha,w\lambda\rangle = 0$.
Then we have $M_\lambda(w,\mu)\simeq M_\lambda(w,s_\alpha\mu)\simeq M_\lambda(s_\alpha w,\mu)\simeq M_\lambda(s_\alpha w,s_\alpha\mu)$.
\end{lem}
\begin{proof}
If necessary, replacing $w$ by $s_\alpha w$, we may assume that $s_\alpha w < w$.
By applying the condition (A\ref{enum:w <-> w lambda condition}) as $w' = s_{w^{-1}\alpha}$, we get $M_\lambda(s_\alpha w,\mu)\simeq M_\lambda(w,\mu)$ and $M_\lambda(s_\alpha w,s_\alpha \mu)\simeq M_\lambda(w,s_\alpha \mu)$.
If $\langle\alpha,\mu\rangle\ge 0$, then $M_\lambda(w,\mu) \simeq M_\lambda(s_\alpha w,s_\alpha \mu)$ by the condition (A\ref{enum:isom condition}).
If $\langle\alpha,\mu\rangle\le 0$, then $M_\lambda(w,s_\alpha \mu)\simeq M_\lambda(s_\alpha w, \mu)$ by the condition (A\ref{enum:isom condition}).
Hence we have $M_\lambda(w,\mu)\simeq M_\lambda(s_\alpha w,\mu)\simeq M_\lambda(w,s_\alpha \mu)\simeq M_\lambda(s_\alpha w,s_\alpha \mu)$ for all $\mu$.
\end{proof}

\begin{lem}\label{lem:from verma mdoule}
Let $w_2\in W$, $w_2 = s_1\dotsm s_l$ be a reduced expression and $\mu_1,\mu_2\in\mathcal{D}$.
Then the following conditions are equivalent.
\begin{enumerate}
\item $\Hom_\mathcal{C}(M_\lambda(w_\lambda,\mu_1),M_\lambda(w_2,\mu_2))\ne 0$.
\item There exists $k\in\Z_{\ge 0}$ such that $\Ext^k_\mathcal{C}(M_\lambda(w_\lambda,\mu_1),M_\lambda(w_2,\mu_2))\ne 0$.
\item $\mu_1\in w_\lambda  W_{\lambda}^0 w_2^{-1} A_{(s_1,\dots,s_l)}(\mu_2)$.
\end{enumerate}
\end{lem}
\begin{proof}
Obviously, (1) implies (2).
We prove the lemma by induction on $\ell_\lambda(w_2)$.
If $w_2 = e$, then the lemma follows from the conditions (A\ref{enum:basis step conditoin - Hom}) and (A\ref{enum:basis step conditoin - Ext}).

Assume that $\ell_\lambda(w_2) > 0$.
Take $\alpha\in\Pi_\lambda$ such that $s_1 = s_\alpha$.
First assume that $\langle \alpha,w_2\lambda\rangle = 0$.
Then we have $W_\lambda^0(s_\alpha w_2)^{-1}A_{(s_2,\dots,s_l)}(\mu_2) = W_\lambda^0(s_\alpha w_2)^{-1}A_{(s_2,\dots,s_l)}(s_\alpha \mu_2)$ by Lemma~\ref{lem:<alpha wlambda> = 0} and induction hypothesis.
By the definition, we have $A_{(s_\alpha)}(\mu_2) = \{\mu_2\}$ or $A_{(s_\alpha)}(\mu_2) = \{\mu_2,s_\alpha\mu_2\}$.
Therefore $W_\lambda^0w_2^{-1}A_{(s_1,\dots,s_l)}(\mu_2) = W_\lambda^0(s_\alpha w_2)^{-1} A_{(s_2,\dots,s_l)}(\mu_2)$ by Lemma~\ref{lem:sum of the set A}.
This implies the lemma in the case of $\langle \alpha,w_2\lambda\rangle = 0$.

In the rest of this proof, we assume that $\langle \alpha,w_2\lambda\rangle \ne 0$.
Assume that $\langle\check{\alpha},\mu_2\rangle\not\in\Z_{<0}$, then, by the condition (A\ref{enum:isom condition}), $M_\lambda(w_2,\mu_2)\simeq M_\lambda(s_\alpha w_2,s_\alpha\mu_2)$.
Since $A_{(s_\alpha)}(\mu_2) = \{\mu_2\}$, we have $w_2^{-1}A_{(s_1,\dots,s_l)}(\mu_2) = (s_\alpha w_2)^{-1}A_{(s_2,\dots,s_l)}(s_\alpha\mu_2)$ by Lemma~\ref{lem:sum of the set A}.
Hence (1)--(3) are equivalent in this case.

Finally assume that $\langle\check{\alpha},\mu_2\rangle\in\Z_{<0}$.
Then we have $A_{(s_\alpha)}(\mu_2) = \{\mu_2,s_\alpha\mu_2\}$.
This implies that $w_2^{-1}A_{(s_1,\dots,s_l)}(\mu_2) =  (s_\alpha w_2)^{-1}A_{(s_2,\dots,s_l)}(\mu_2)\cup  (s_\alpha w_2)^{-1}A_{(s_2,\dots,s_l)}(s_\alpha\mu_2)$ by Lemma~\ref{lem:sum of the set A}.
By the induction hypothesis, $\mu_1\not\in w_\lambda  W_{\lambda}^0 w_2^{-1} A_{(s_1,\dots,s_l)}(\mu_2)$ if and only if 
\[
	\Hom_\mathcal{C}(M_\lambda(w_\lambda,\mu),M_\lambda(s_\alpha w_2,\mu_2)) = \Hom_\mathcal{C}(M_\lambda(w_\lambda,\mu),M_\lambda(s_\alpha w_2,s_\alpha \mu_2)) = 0.
\]
From the condition (A\ref{enum:exact sequence condition}), we have an exact sequence 
\begin{equation}\label{eq:4-term exact sequence}
0\to M_\lambda(s_\alpha w_2,\mu_2)\to M_\lambda(w_2,\mu_2)\to M_\lambda(s_\alpha w_2,s_\alpha\mu_2)\to M_\lambda(s_\alpha w_2,\mu_2)\to 0.
\end{equation}
Since the functor $\Hom_\mathcal{C}$ is left-exact, we have an exact sequence 
\begin{multline*}
0\longrightarrow \Hom_\mathcal{C}(M_\lambda(w_\lambda,\mu),M_\lambda(s_\alpha w_2,\mu_2))\\\longrightarrow \Hom_\mathcal{C}(M_\lambda(w_\lambda,\mu),M_\lambda(w_2,\mu_2))\longrightarrow \Hom_\mathcal{C}(M_\lambda(w_\lambda,\mu),M_\lambda(s_\alpha w_2,s_\alpha\mu_2)).
\end{multline*}
If (3) does not hold, $\Hom_\mathcal{C}(M_\lambda(w_\lambda,\mu),M_\lambda(s_\alpha w_2,\mu_2)) = \Hom_\mathcal{C}(M_\lambda(w_\lambda,\mu),M_\lambda(s_\alpha w_2,s_\alpha \mu_2)) = 0$.
Hence we have $\Hom_\mathcal{C}(M_\lambda(w_\lambda,\mu),M_\lambda(w_2,\mu_2)) = 0$, i.e., (1) does not hold.
Therefore, (1) implies (3).

Now assume that (1) does not hold, i.e., $\Hom_\mathcal{C}(M_\lambda(w_\lambda,\mu),M_\lambda(w_2,\mu_2)) = 0$.
We prove $\Hom_\mathcal{C}(M_\lambda(w_\lambda,\mu),M_\lambda(s_\alpha w_2,\mu_2)) = \Hom_\mathcal{C}(M_\lambda(w_\lambda,\mu),M_\lambda(s_\alpha w_2,s_\alpha\mu_2)) = 0$ and, for all $k\in\Z_{>0}$, $\Ext_\mathcal{C}^k(M_\lambda(w_\lambda,\mu),M_\lambda(w_2,\mu_2)) = 0$.
These imply the lemma.

By the exact sequence 
\[
	0\to \Hom_\mathcal{C}(M_\lambda(w_\lambda,\mu),M_\lambda(s_\alpha w_2,\mu_2))\to \Hom_\mathcal{C}(M_\lambda(w_\lambda,\mu),M_\lambda(w_2,\mu_2)),
\]
we have $\Hom_\mathcal{C}(M_\lambda(w_\lambda,\mu),M_\lambda(s_\alpha w_2,\mu_2)) = 0$.
Hence we have $\Ext_\mathcal{C}^k(M_\lambda(w_\lambda,\mu),M_\lambda(s_\alpha w_2,\mu_2)) = 0$ for all $k\in\Z_{\ge 0}$ by induction hypothesis.
Put $L = \Ker(M_\lambda(s_\alpha w_2,s_\alpha \mu_2)\to M_\lambda(s_\alpha w_2,\mu_2))$.
From an exact sequence \eqref{eq:4-term exact sequence}, we have exact sequences 
\[
	0\to M_\lambda(s_\alpha w,\mu_2)\to M_\lambda(w,\mu_2)\to L\to 0
\]
and
\[
	0\to L\to M_\lambda(s_\alpha w_2,s_\alpha \mu_2)\to M_\lambda(s_\alpha w_2,\mu_2)\to 0.
\]
Using $\Ext_\mathcal{C}^k(M_\lambda(w_\lambda,\mu),M_\lambda(s_\alpha w_2,\mu_2)) = 0$ and the long exact sequences induced from there sequences, we have
\[
	\Ext_\mathcal{C}^k(M_\lambda(w_\lambda,\mu),M_\lambda(w_2,\mu_2))\simeq \Ext_\mathcal{C}^k(M_\lambda(w_\lambda,\mu),L)\simeq \Ext_\mathcal{C}^k(M_\lambda(w_\lambda,\mu),M_\lambda(s_\alpha w_2,s_\alpha\mu_2)).
\]
In particular, $\Hom_\mathcal{C}(M_\lambda(w_\lambda,\mu),M_\lambda(s_\alpha w_2,s_\alpha\mu_2)) \simeq \Hom_\mathcal{C}(M_\lambda(w_\lambda,\mu),M_\lambda(w_2,\mu_2)) = 0$.
By induction hypothesis, we have $\Ext_\mathcal{C}^k(M_\lambda(w_\lambda,\mu),M_\lambda(s_\alpha w_2,s_\alpha\mu_2)) = 0$ for all $k\in\Z_{>0}$.
Hence we have $\Ext_\mathcal{C}^k(M_\lambda(w_\lambda,\mu),M_\lambda(w_2,\mu_2)) = 0$ for all $k\in\Z_{>0}$.
\end{proof}
If for some abelian category $\mathcal{C}$ and some regular $\lambda$ there exist objects which satisfy the conditions (A1--5), then the set $A_{(s_1,\dots,s_l)}(\mu)$ is independent of the choice of a reduced expression by Lemma~\ref{lem:from verma mdoule}.
In the rest of this section, we assume it (It will be proved in Section~\ref{sec:Proof of the main theorems}).
Put $A_{w_2}(\mu) = A_{(s_1,\dots,s_l)}(\mu)$.

\begin{thm}\label{thm:main thm, general theory}
Let $w_1,w_2\in W_\lambda$ and $\mu_1,\mu_2\in\mathcal{D}$.
The following conditions are equivalent.
\begin{enumerate}
\item $\Hom_\mathcal{C}(M_\lambda(w_1,\mu_1),M_\lambda(w_2,\mu_2))\ne 0$.
\item There exists $k\in\Z_{\ge 0}$ such that $\Ext^k_\mathcal{C}(M_\lambda(w_1,\mu_1),M_\lambda(w_2,\mu_2))\ne 0$.
\item $w_1^{-1}w_\lambda A_{w_\lambda w_1}(w_\lambda\mu_1)\cap W_\lambda^0w_2^{-1}A_{w_2}(\mu_2)\ne \emptyset$.
\end{enumerate}
\end{thm}
\begin{proof}
We prove by backward induction on $\ell_\lambda(w_1)$.
If $w_1 = w_\lambda$, then from Lemma~\ref{lem:from verma mdoule}, (1)--(3) are equivalent.
We use the similar argument in the proof of Lemma~\ref{lem:from verma mdoule}.

Take $\alpha\in\Pi_\lambda$ such that $s_\alpha w_1 > w_1$.
Put $\beta = -w_\lambda(\alpha)\in\Pi_\lambda$.
We have $A_{w_\lambda w_1}(w_\lambda\mu_1) = \bigcup_{\mu_0\in A_{s_\beta}(w_\lambda\mu_1)}s_\beta A_{w_\lambda s_\alpha w_1}(s_\beta\mu_0)$ by Lemma~\ref{lem:sum of the set A}.
First assume that $\langle\alpha,w_1\lambda\rangle = 0$.
Then by Lemma~\ref{lem:<alpha wlambda> = 0}, we have $M_\lambda(w_1,\mu_1) \simeq M_\lambda(s_\alpha w_1,\mu_1)\simeq M_\lambda(s_\alpha w_1,s_\alpha \mu_1)$.
This implies the lemma.

In the rest of this proof, we assume that $\langle\alpha,w_1\lambda\rangle \ne 0$.
First assume that $\langle \check{\alpha},\mu_1\rangle\not\in\Z_{>0}$, then by the condition (A\ref{enum:isom condition}), $M_\lambda(w_1,\mu_1)\simeq M_\lambda(s_\alpha w_1,s_\alpha\mu_1)$.
Since $A_{s_\beta}(w_\lambda\mu_1) = \{w_\lambda\mu_1\}$, $A_{w_\lambda w_1}(w_\lambda\mu_1) = A_{w_\lambda s_\alpha w_1}(w_\lambda s_\alpha \mu_0)$.
Hence we have the lemma.

Finally, we assume that $\langle \check{\alpha},\mu_1\rangle\in\Z_{>0}$.
We have $w_1^{-1}w_\lambda A_{w_\lambda w_1}(w_\lambda\mu_1)\cap W_\lambda^0w_2^{-1}A_{w_2}(\mu_2)\ne \emptyset$ if and only if $(s_\alpha w_1)^{-1}w_\lambda A_{w_\lambda s_\alpha w_1}(w_\lambda\mu_1)\cap W_\lambda^0w_2^{-1}A_{w_2}(\mu_2)\ne \emptyset$ or $(s_\alpha w_1)^{-1}w_\lambda A_{w_\lambda s_\alpha w_1}(w_\lambda s_\alpha \mu_1)\cap W_\lambda^0w_2^{-1}A_{w_2}(\mu_2)\ne \emptyset$ since $A_{s_\beta}(w_\lambda\mu_1) = \{w_\lambda\mu_1,w_\lambda s_\alpha\mu_1\}$.
By the condition (A\ref{enum:isom condition}), we have $M_\lambda(w_1,s_\alpha\mu_1)\simeq M_\lambda(s_\alpha w_1,\mu_1)$.
Hence, there exists an exact sequence $0\to M_\lambda(s_\alpha w_1,\mu_1)\to M_\lambda(s_\alpha w_1,s_\alpha \mu_1)\to M_\lambda(w_1,\mu_1)\to M_\lambda(s_\alpha w_1,\mu_1)\to 0$ by the condition (A\ref{enum:exact sequence condition}).
Therefore (1) implies (3).

Now assume (1) dose not hold, i.e., $\Hom_\mathcal{C}(M_\lambda(w_1,\mu_1),M_\lambda(w_2,\mu_2)) = 0$.
We prove that $\Hom_\mathcal{C}(M_\lambda(s_\alpha w_1,\mu_1),M_\lambda(w_2,\mu_2)) = \Hom_\mathcal{C}(M_\lambda(s_\alpha w_1,s_\alpha \mu_1),M_\lambda(w_2,\mu_2)) = 0$ and, for all $k\in\Z_{>0}$, $\Ext^k_\mathcal{C}(M_\lambda(w_1,\mu_1),M_\lambda(w_2,\mu_2)) = 0$.
Since we have an exact sequence
\[
	0\longrightarrow \Hom_\mathcal{C}(M_\lambda(s_\alpha w_1,\mu_1),M_\lambda(w_2,\mu_2))\to \Hom_\mathcal{C}(M_\lambda(w_1,\mu_1),M_\lambda(w_2,\mu_2)),
\]
we have $\Hom_\mathcal{C}(M_\lambda(s_\alpha w_1,\mu_1),M_\lambda(w_2,\mu_2)) = 0$.
Hence, by induction hypothesis, we have that $\Ext_\mathcal{C}^k((M_\lambda(s_\alpha w_1,\mu_1),M_\lambda(w_2,\mu_2)) = 0$.
Therefore we have 
\[
	\Ext^k_\mathcal{C}((M_\lambda(w_1,\mu_1),M_\lambda(w_2,\mu_2))\simeq \Ext_\mathcal{C}^k(M_\lambda(s_\alpha w_1,s_\alpha \mu_1),M_\lambda(w_2,\mu_2)).
\]
In particular, $\Hom_\mathcal{C}(M_\lambda(s_\alpha w_1,s_\alpha \mu_1),M_\lambda(w_2,\mu_2)) = 0$.
Hence we have 
\[
	\Ext_\mathcal{C}^k(M_\lambda(w_1,\mu_1),M_\lambda(w_2,\mu_2))\simeq \Ext^k_\mathcal{C}(M_\lambda(s_\alpha w_1,s_\alpha \mu_1),M_\lambda(w_2,\mu_2)) = 0.
\]
\end{proof}

\section{Proof of the main theorems}\label{sec:Proof of the main theorems}
In this section, we prove Theorem~\ref{thm:main theorem, principal series} and Theorem~\ref{thm:main theorem, twisted Verma modules} using the result of Section~\ref{sec:General theory}.
First we consider the twisted Verma modules.
Fix a regular dominant integral element $\lambda$.
Put $\mathcal{C} = \mathcal{O}$, $\mathcal{D} = \mathfrak{h}^*$.
Set $M_\lambda(w,\mu) = T_{w^{-1}w_0}M(w_0\mu)$.
\begin{lem}\label{lem:conditions for twisted Verma module}
The modules $\{M_\lambda(w,\mu)\}$ satisfy the conditions (A1--5).
\end{lem}
\begin{proof}
The condition (A\ref{enum:w <-> w lambda condition}) is obvious since $\lambda$ is regular.
The conditions (A\ref{enum:isom condition}) and (A\ref{enum:exact sequence condition}) are \cite[Proposition~6.3]{MR1985191}.
Since $T_{w_0}M(w_0\mu)\simeq \delta M(\mu)$~\cite[Corollary~5.1]{MR1985191}, we have $\Hom_\mathcal{O}(M_\lambda(w_\lambda,\mu'),M_\lambda(e,\mu)) = \Hom_\mathcal{O}(M(w_0\mu'),\delta M(\mu))$.
Since $\Hom_\mathcal{O}(M(w_0\mu'),\delta M(\mu)) \ne 0$ if and only if $w_0\mu' = \mu$, we have (A\ref{enum:basis step conditoin - Hom}).
Moreover, we have $\bigoplus_{\mu'\in\mathfrak{h}^*}\Ext^k_\mathcal{O}(M(w_0\mu'),\delta M(\mu))\simeq H^k(\mathfrak{n},\delta M(\mu))\simeq H_k(\overline{\mathfrak{n}},M(\mu)) = 0$ where $\overline{\mathfrak{n}}$ is the nilradical of the opposite Borel subalgebra of $\mathfrak{b}$.
Hence we have (A\ref{enum:basis step conditoin - Ext}).
\end{proof}
From Lemma~\ref{lem:conditions for twisted Verma module} and Theorem~\ref{thm:main thm, general theory}, we have Theorem~\ref{thm:main theorem, twisted Verma modules}.

Next, we consider the principal series representations of $G$.
This is a full-subcategory of $\mathfrak{g}\oplus\mathfrak{g}$-modules.
We also regard $\mathcal{H}$ as a full-subcategory of $\mathfrak{g}$-bimodules.
\begin{lem}\label{lem:isom between principal series}
Let $\lambda,\mu\in\mathfrak{h}^*$ such that $\lambda - \mu\in\mathcal{P}$, $w\in W$.
Put $\Delta^- = -\Delta^+$ and $\Delta^-_\lambda = -\Delta_\lambda^+$.
\begin{enumerate}
\item There exists $w'\in W_\lambda$ such that $\Delta^+\cap (w'w^{-1})^{-1}\Delta^-\cap w\Delta_\lambda = \emptyset$.
\item Take $w'$ as in (1).
Then we have $L(M(w\lambda),\delta M(w\mu))\simeq L(M(w'\lambda),\delta M(w'\mu))$.
\end{enumerate}
\end{lem}
\begin{proof}
(1)
Since $w^{-1}\Delta^+\cap \Delta_\lambda$ is a positive system of $\Delta_\lambda$, there exists $w'\in W_\lambda$ such that $w^{-1}\Delta^+\cap \Delta_\lambda = (w')^{-1}\Delta_\lambda^+$.
Since $(w')^{-1}\Delta_\lambda^- = (w')^{-1}(\Delta^-\cap \Delta_\lambda) = (w')^{-1}\Delta^-\cap \Delta_\lambda$, we have $\Delta^+\cap (w'w^{-1})^{-1}\Delta^-\cap w\Delta_\lambda = w(w^{-1}\Delta^+\cap (w')^{-1}\Delta^-\cap \Delta_\lambda) = w(w^{-1}\Delta^+\cap (w')^{-1}\Delta^-_\lambda\cap \Delta_\lambda) = \emptyset$.

(2)
By the condition of $w'$, for all $\alpha\in \Delta^+\cap (w'w^{-1})^{-1}\Delta^-$ we have $\langle\check{\alpha},-w\lambda\rangle \not\in\Z$.
Hence by \cite[4.8.~Proposition]{MR0430005} we have $L(M(w\lambda),\delta M(w\mu))\simeq L(M(w'\lambda),\delta M(w'\mu))$.
\end{proof}

By Lemma~\ref{lem:isom between principal series}, it is sufficient to study $\Hom_\mathcal{H}(L(M(w'\lambda),\delta M(\mu')),L(M(w\lambda),\delta M(\mu)))$ for dominant $\lambda$ and $w,w'\in W_\lambda$.
Moreover, we may assume $\mu\in W_\lambda\mu'$ since $L(M(w\lambda),\delta M(\mu)) = 0$ unless $w\lambda - \mu\in\mathcal{P}$.
Fix such a $\lambda$ and put $M_\lambda(w,\mu) = L(M(w\lambda),\delta M(\mu))$ for $\mu\in \lambda + \mathcal{P}$ and $w\in W_\lambda$.
Put $\mathcal{D} = \lambda + \mathcal{P}$.
For $\alpha\in\Pi_\lambda$, let $C_\alpha$ be Joseph's Enright functor~\cite{MR664114}.
Recall that $M\in\mathcal{O}$ is called $\alpha$-free if the canonical map $M\to C_\alpha M$ is injective.

\begin{lem}\label{lem:conditions for principal series}
Let $\mu\in\lambda + \mathcal{P}$ and $\alpha\in\Pi_\lambda$.
\begin{enumerate}
\item If $N\in\mathcal{O}$ is $\alpha$-free and $\langle\check{\alpha},\mu\rangle\in\Z_{\le 0}$ then $L(M_\lambda(s_\alpha\mu),C_\alpha N)\simeq L(M_\lambda(\mu),N)$.
\item Let $w\in W_\lambda$. If $\langle\check{\alpha},w\lambda\rangle\in\Z_{\le 0}$ and $\langle\check{\alpha},\mu\rangle\in\Z_{\le 0}$, then $L(M(s_\alpha w\lambda),M(s_\alpha\mu))\simeq L(M(w\lambda),M(\mu))$.
\item Let $w\in W_\lambda$. If $\langle\check{\alpha},w\lambda\rangle\in\Z_{\le 0}$ and $\langle\check{\alpha},\mu\rangle\in\Z_{\ge 0}$, then $L(M(s_\alpha w\lambda),\delta M(s_\alpha\mu))\simeq L(M(w\lambda),\delta M(\mu))$.
\item We have $L(M(w_\lambda\lambda),\delta M(\mu))\simeq L(M(\lambda),M(w_\lambda\mu))$.
\item The modules $\{M_\lambda(w,\mu)\}$ satisfy the conditions (A1--5).
\end{enumerate}
\end{lem}
\begin{proof}
(1)
Put $M = M(\mu)$ and $M' = M(s_\alpha\mu)$ in \cite[3.8.~Lemma]{MR664114}.
Then we get (1).

(2)
Take $N = M(\mu)$ in (1) and use \cite[2.5.~Lemma]{MR664114}.

(3)
Let $\widetilde{\lambda}\in \lambda + \mathcal{P}$ be a regular element such that $\widetilde{\lambda}$ is dominant.
Then by~\cite[2.5.~Lemma]{MR733462}, we have $C_\alpha \delta M(\widetilde{\lambda}) \simeq \delta M(s_\alpha\widetilde{\lambda})$.
For $\mathfrak{g}\oplus\mathfrak{g}$-module $N$, let $N^\eta$ be a $\mathfrak{g}\oplus\mathfrak{g}$-module where the action is twisted by $(X,Y)\mapsto (Y,X)$.
Using \cite[2.8]{MR664114}, we have $L(M(\widetilde{\lambda}),C_\alpha\delta M(\mu))\simeq L(M(s_\alpha\widetilde{\lambda}),\delta M(\mu))\simeq L(M(\mu),\delta M(s_\alpha\widetilde{\lambda}))^\eta \simeq L(M(\mu),C_\alpha \delta M(\widetilde{\lambda}))^\eta$.
Notice that $\delta M(s_\alpha\widetilde{\lambda})$ is $\alpha$-free.
Hence we have $L(M(\mu),C_\alpha \delta M(\widetilde{\lambda}))^\eta\simeq L(M(s_\alpha\mu),\delta M(\widetilde{\lambda}))^\eta \simeq L(M(\widetilde{\lambda}),\delta M(s_\alpha\mu))$ by (1).
Therefore we have $C_\alpha \delta M(\mu)\simeq \delta M(s_\alpha\mu)$.
We get (3) by (1).

(4)
Take $w\in W_\lambda$ such that $\langle \check{\beta},w\mu\rangle \in\Z_{\le 0}$ for all $\beta\in\Delta^+_\lambda$.
Put $\mu_0 = w\mu$.
Let $w = s_{\alpha_l}\dotsm s_{\alpha_1}$ be a reduced expression.
Then we have $\langle \check{\alpha_i},s_{\alpha_i}\dots s_{\alpha_l}\mu\rangle\in\Z_{\ge 0}$ and $\langle \check{\alpha_i},s_{\alpha_{i - 1}}\dots s_{\alpha_1}w_\lambda\lambda\rangle\in\Z_{\le 0}$.
Hence by (3), we have $L(M(w_\lambda\lambda),\delta M(\mu))\simeq L(M(ww_\lambda\lambda),\delta M(\mu_0))$.
Take a reduced expression of $ww_\lambda$ and use (2), then we have $L(M(ww_\lambda\lambda),M(\mu_0))\simeq L(M(\lambda),M(w_\lambda \mu))$ by the same argument.
Since $\delta M(\mu_0)\simeq M(\mu_0)$, we have (4).

(5)
The condition (A\ref{enum:w <-> w lambda condition}) is obvious.
The condition (A\ref{enum:isom condition}) follows from (2) and (A\ref{enum:exact sequence condition}) from \cite[4.7.~Corollary]{MR664114}.
To prove (A\ref{enum:basis step conditoin - Hom}) and (A\ref{enum:basis step conditoin - Ext}), we may assume that $\mu'\in W_\lambda\mu$.
Let $\mu_1\in W_\lambda\mu$ such that $\langle\beta,\mu_1\rangle\ge 0$ for all $\beta\in \Delta_\lambda^+$.
Take $w,w'\in W_\lambda$ such that $\mu' = w'\mu_1$ and $\mu = w\mu_1$.
Then by the argument in (4), we have $L(M(w_\lambda\lambda),\delta M(\mu'))\simeq L(M(w_\lambda (w')^{-1}w_\lambda \lambda),\delta M(w_\lambda \mu_1))$ and $L(M(\lambda),\delta M(\mu))\simeq L(M(w^{-1}\lambda),\delta M(\mu_1))$.
We prove (A\ref{enum:basis step conditoin - Hom}) and (A\ref{enum:basis step conditoin - Ext}).
First we assume that $\mu_1$ is regular.
Then by (4), we have 
\begin{align*}
	&\Ext^k_\mathcal{H}(M_\lambda(w_\lambda,\mu'),M_\lambda(e,\mu))\\
	\simeq& \Ext^k_\mathcal{H}(L(M(w_\lambda (w')^{-1}w_\lambda \lambda),\delta M(w_\lambda \mu_1)),L(M(w^{-1}\lambda),\delta M(\mu_1)))\\
	\simeq& \Ext^k_\mathcal{H}(L(M(w_\lambda (w')^{-1}w_\lambda \lambda),\delta M(w_\lambda \mu_1))^\eta,L(M(w^{-1}\lambda),\delta M(\mu_1))^\eta)\\
	\simeq& \Ext^k_\mathcal{H}(L(M(w_\lambda \mu_1), \delta M(w_\lambda (w')^{-1}w_\lambda \lambda)),L(M(\mu_1),\delta M(w^{-1}\lambda)))\\
	\simeq& \Ext^k_\mathcal{H}(L(M(\mu_1), M((w')^{-1}w_\lambda\lambda)),L(M(\mu_1),\delta M(w^{-1}\lambda)))
\end{align*}
By the Bernstein-Gelfand-Joseph-Enright equivalence~\cite[5.9.~Theorem]{MR581584}, this space is isomorphic to $\Ext^k(M((w')^{-1}w_\lambda\lambda),\delta M(w^{-1}\lambda))$.
Hence the proof is done in this case (see the proof of Lemma~\ref{lem:conditions for principal series}).

We prove (A\ref{enum:basis step conditoin - Hom}) and (A\ref{enum:basis step conditoin - Ext}) for general $\mu_1$.
Take a regular element $\mu_2\in \mu_1 + \mathcal{P}$ such that for all $\beta\in\Delta^+_\lambda$.
Let $T_{\mu_1}^{\mu_2}$ be the translation functor of $\mathcal{O}$ and $L_{\mu_1}^{\mu_2}$ the translation functor of $\mathcal{H}$ with respect to the left $\mathfrak{g}$-action.
Then we have $L_{\mu_1}^{\mu_2}L(M,N) = L(M,T_{\mu_1}^{\mu_2}N)$ for $M,N\in \mathcal{O}$.
Since $T_{\mu_1}^{\mu_2}$ commutes with $\delta$, we have
\begin{align*}
\Ext^k_\mathcal{H}(M_\lambda(w_\lambda,\mu'),M_\lambda(e,\mu))
&\simeq \Ext^k_\mathcal{H}(L(M(w_\lambda \lambda),T_{\mu_1}^{\mu_2} M(w'\mu_2)),L(M(\lambda),T_{\mu_1}^{\mu_2}\delta M(w\mu_2)))\\
&\simeq \Ext^k_\mathcal{H}(L_{\mu_1}^{\mu_2}L(M(w_\lambda \lambda), M(w'\mu_2)),L(M(\lambda),T_{\mu_1}^{\mu_2}\delta M(w\mu_2)))\\
&\simeq \Ext^k_\mathcal{H}(L(M(w_\lambda \lambda), M(w'\mu_2)),L_{\mu_2}^{\mu_1}L(M(\lambda),T_{\mu_1}^{\mu_2}\delta M(w\mu_2)))\\
&\simeq \Ext^k_\mathcal{H}(L(M(w_\lambda \lambda),M(w'\mu_2)),L(M(\lambda),\delta T_{\mu_2}^{\mu_1}T_{\mu_1}^{\mu_2}M(w\mu_2)))
\end{align*}
The module $T_{\mu_2}^{\mu_1}T_{\mu_1}^{\mu_2}M(w\mu_2)$ has a filtration $0 ~= M_0\subset M_1\subset \dots\subset M_r = T_{\mu_2}^{\mu_1}T_{\mu_1}^{\mu_2}M(w\mu_2)$ such that $\{M_i/M_{i - 1}\mid 1\le i\le r\} = \{M(wv\mu_2)\mid v\in W_{\mu_1}^0\}$~\cite[2.3 Satz (b)]{MR552943}.
Since $\lambda$ is dominant, $L(M(\lambda),\cdot)$ is an exact functor.
Hence we have an exact sequence $0\to L(M(\lambda),M_i)\to L(M(\lambda),M_{i - 1})\to L(M(\lambda),M_i/M_{i - 1}) \to 0$.
Using the long exact sequence and the result in regular case, we have $\Ext^k_\mathcal{H}(M_\lambda(w_\lambda,\mu'),M_\lambda(e,\mu)) = 0$ for $k>0$.
Moreover, by the vanishing of the $\Ext$-groups, 
\begin{multline*}
\dim \Hom_\mathcal{H}(M_\lambda(w_\lambda,\mu'),M_\lambda(e,\mu)) \\
= \sum_{v\in W_{\mu_1}^0}\dim\Hom^k_\mathcal{H}(L(M(w_\lambda \lambda),M(w'\mu_2)),L(M(\lambda),\delta M(wv\mu_2))).
\end{multline*}
From this formula, we have $\Hom_\mathcal{H}(M_\lambda(w_\lambda,\mu'),M_\lambda(e,\mu))\ne 0$ if and only if $w'\in W_\lambda^0w_\lambda wv$ for some $v\in W_{\mu_1}^0$.
This condition is equivalent to $\mu' = w'\mu_1 \in W_\lambda^0w_\lambda w\mu_1 = W_\lambda^0w_\lambda \mu$.
\end{proof}
From Lemma~\ref{lem:conditions for principal series} and Theorem~\ref{thm:main thm, general theory}, we have Theorem~\ref{thm:main theorem, principal series}.

\def\cprime{$'$} \def\dbar{\leavevmode\hbox to 0pt{\hskip.2ex \accent"16\hss}d}

\end{document}